\documentclass[11pt,a4paper, final, twoside]{article}
\usepackage{amsmath}
\usepackage{fancyhdr}
\usepackage{amsthm}
\usepackage{amsfonts}
\usepackage{amssymb}
\usepackage{amscd}
\usepackage{latexsym}
\usepackage{amsthm}
\usepackage{graphicx}
\usepackage{graphics}
\usepackage{afterpage}
\usepackage[colorlinks=true, urlcolor=blue,  linkcolor=black, citecolor=black]{hyperref}
\usepackage{color}
\newcommand{\HRule}{\rule{\linewidth}{1pt}}

\newtheorem{thm}{Theorem}[section]

\newtheorem{note}[thm]{Note}

\theoremstyle{proposition}

\newtheorem{cor}[thm]{Corollary}

\theoremstyle{definition}
\newtheorem{defn}{Definition}[section]
\theoremstyle{remark}
\newtheorem{rem}{Remark}[section]

\numberwithin{equation}{section}
\begin{document}
\hyphenpenalty=100000

\begin{flushright}

{\Large \textbf{\\ The general subclasses of the analytic functions and their various properties }}\\[5mm]
{\large \textbf{Sercan Topkaya$^\mathrm{*1}$\footnote{\emph{*Corresponding author: E-mail: topkaya.sercan@hotmail.com}}  and Nizami Mustafa$^\mathrm{1}$}}\\[1mm]
$^\mathrm{1}${\footnotesize \it Department of Mathematics, Faculty of Science and Letters,\\ Kafkas University, Campus, 36100, Kars-Turkey\\ \textit{topkaya.sercan@hotmail.com} and \textit{nizamimustafa@gmail.com}}\\ 
\end{flushright}

\begin{flushleft}\fbox{%
\begin{minipage}{1.3in}
{\slshape \textbf{Original Research Article}\/}
\end{minipage}}
\end{flushleft}

\begin{flushright}\footnotesize \it Received: XX December 20XX\\ 
Accepted: XX December 20XX\\
Online Ready: XX December 20XX
\end{flushright}
\HRule\\[3mm]

{\Large \textbf{Abstract}}\\[4mm]
\fbox{%
\begin{minipage}{5.4in}{\footnotesize The object of the present paper is to introduce and investigate two new general subclasses 
${{S}^{*}}C(\alpha ,\beta ;\gamma )$ and $T{{S}^{*}}C(\alpha ,\beta ;\gamma )~~(\alpha, \beta \in [0,1),~\gamma \in [0,1])$ of the analytic functions. Here, we give sufficient conditions as well as necessary and sufficient conditions for the functions belonging to the classes.} \end{minipage}}\\[1mm]

\afterpage{
\fancyhead{} \fancyfoot{} 
\fancyfoot[R]{\footnotesize\thepage}
\fancyhead[R]{\scriptsize\it Asian Research Journal of Mathematics 
{{X(X), XX--XX}},~20XX \\
 }}

\section{\large Introduction and preliminaries}
Let $A$ be the class of analytic functions$f(z)$ in the open unit disk $U=\left\{ z\in \mathbb{C}:\text{ }\left| z \right|<1 \right\}$ of the form
\begin{equation}\label{eq11}
f(z)=z+{{a}_{2}}{{z}^{2}}+{{a}_{3}}{{z}^{3}}+\cdot \cdot \cdot +{{a}_{n}}{{z}^{n}}+\cdot \cdot \cdot =z+\sum\limits_{n=2}^{\infty}{{{a}_{n}}{{z}^{n}}},\text{ }{{a}_{n}}\in \mathbb{C}.
\end{equation}
Also, by $S$we will denote the family of all functions in $A$ which are univalent in $U$. 
Let $T$ denote the subclass of all functions $f(z)$ in $A$ of the form
\begin{equation}\label{eq12}
f(z)=z-{{a}_{2}}{{z}^{2}}-{{a}_{3}}{{z}^{3}}-\cdot \cdot \cdot -{{a}_{n}}{{z}^{n}}-\cdot \cdot \cdot =z-\sum\limits_{n=2}^{\infty }{{{a}_{n}}{{z}^{n}}},\text{ }{{a}_{n}}\ge 0.
\end{equation}
Some of the important and well-investigated subclasses of the univalent functions class $S$ include the classes ${{S}^{*}}(\alpha )\text{ and }C(\alpha )$, respectively, starlike and convex of order $\alpha $ $\left( \alpha \in \left[ 0,1 \right) \right)$ in the open unit disk $U$. \\
By definition, we have (see for details, \cite{Duren,Goodman}, also \cite{Srivastava})    
$${{S}^{*}}(\alpha )=\left\{ f\in A:\text{ Re}\left( \frac{z{f}'(z)}{f(z)} \right)>\alpha ,\text{ }z\in U\text{ } \right\},\text{ }\alpha \in \left[ 0,1 \right),$$
and 
$$C(\alpha )=\left\{ f\in A:\text{ Re}\left( 1+\frac{z{f}''(z)}{{f}'(z)} \right)>\alpha ,\text{ }z\in U\text{ } \right\},\text{ }\alpha \in \left[ 0,1 \right) .$$
Note that, we will use $T{{S}^{*}}(\alpha )={{S}^{*}}(\alpha )\cap T\text{ and }TC(\alpha )=C(\alpha )\cap T.$ \\
Interesting generalization of the functions classes ${{S}^{*}}(\alpha )\text{ and }C(\alpha )$, are classes ${{S}^{*}}(\alpha ,\beta )$ and $C(\alpha ,\beta )$, which defined by
$${{S}^{*}}(\alpha ,\beta )=\left\{ f\in A:\text{ Re}\left( \frac{z{f}'(z)}{\beta z{f}'(z)+(1-\beta )f(z)} \right)>\alpha ,\text{ }z\in U \right\},\text{ }\alpha ,\beta \in \left[ 0,1 \right)$$
and
$$C(\alpha ,\beta )=\left\{ f\in A:\text{ Re}\left( \frac{{f}'(z)+z{f}''(z)}{{f}'(z)+\beta z{f}''(z)} \right)>\alpha ,\text{ }z\in U \right\},\alpha ,\beta \in \left[ 0,1 \right),$$
respectively. \\
We will denote $T{{S}^{*}}(\alpha ,\beta )={{S}^{*}}(\alpha ,\beta )\cap T$ and $TC(\alpha ,\beta )=C(\alpha ,\beta )\cap T$.\\
These classes $T{{S}^{*}}(\alpha ,\beta )$ and $TC(\alpha ,\beta )$ were extensively studied by Altinta\c s and Owa \cite{Alt?ntas88}, Porwal \cite{Porwal14}, and certain conditions for hypergeometric functions and generalized Bessel functions for these classes were studied Moustafa \cite{Moustafa} and Porwal and Dixit \cite{Porwal13}. \\
The coefficient problems for the subclasses $TS^*(\alpha, \beta) $ and $TC(\alpha, \beta) $ were investigated by Alt\i nta\c s and Owa in \cite{Alt?ntas88}. They, also investigated properties like starlike and convexity of these classes.\\
Also, the coefficient problems, representation formula and distortion theorems for these subclasses $S^*(\alpha, \beta, \mu) $ and $C^*(\alpha, \beta, \mu) $ of the analytic functions were given by Owa and Aouf in \cite{Owa2}. \\
In \cite{Ekrem}, results of Silverman were extended by Kadio\u glu.            \\
Inspired by the studies mentioned above, we define a unification of the functions classes ${{S}^{*}}(\alpha ,\beta )$ and $C(\alpha ,\beta )$ as follows.
\begin{defn}\label{tnm11}
A function $f\in A$ given by (\ref{eq11}) is said to be in the class ${{S}^{*}}C(\alpha ,\beta ;\gamma )$, $\alpha ,\beta \in \left[ 0,1 \right),~\gamma \in \left[ 0,1 \right]$ if the following condition is satisfied
$$\operatorname{Re}\left\{ \frac{z{f}'(z)+\gamma {{z}^{2}}{f}''(z)}{\gamma z\left( {f}'(z)+\beta z{f}''(z) \right)+(1-\gamma )\left( \beta z{f}'(z)+(1-\beta )f(z) \right)} \right\}>\alpha ,z\in U.$$
Also, we will denote $T{{S}^{*}}C(\alpha ,\beta ;\gamma )={{S}^{*}}C(\alpha ,\beta ;\gamma )\cap T.$
\end{defn}
In special case, we have: \\ 
${{S}^{*}}C(\alpha ,\beta ;0)={{S}^{*}}(\alpha ,\beta );\text{ }{{S}^{*}}C(\alpha ,\beta ;1)=C(\alpha ,\beta );\text{ }{{S}^{*}}C(\alpha ,0;0)={{S}^{*}}(\alpha );$\\ 
${{S}^{*}}C(\alpha ,0;1)=C(\alpha );~T{{S}^{*}}C(\alpha ,\beta ;0)=T{{S}^{*}}(\alpha ,\beta );\text{ }T{{S}^{*}}C(\alpha ,\beta ;1)=TC(\alpha ,\beta );$\\
$T{{S}^{*}}C(\alpha ,0;0)=T{{S}^{*}}(\alpha ); 
T{{S}^{*}}C(\alpha ,0;1)=TC(\alpha ).$ \\
Suitably specializing the parameters we note that
\begin{itemize}
\item[1)] ${{S}^{*}}C(\alpha ,0;0)=S^*(\alpha )$ \cite{Silverman}
\item[2)] ${{S}^{*}}C(\alpha ,0;1)=C(\alpha )$ \cite{Silverman}
\item[3)]$T{{S}^{*}}C(\alpha ,\beta ;0)=T{{S}^{*}}(\alpha ,\beta )$ \cite{Alt?ntas91,Alt?ntas95,Alt?ntas4} and \cite{Porwal14}
\item[4)] $T{{S}^{*}}C(\alpha ,0;0)=TS^*(\alpha )$ \cite{Silverman}
\item[5)]$T{{S}^{*}}C(\alpha ,\beta ;1)=TC(\alpha ,\beta )$ \cite{Alt?ntas88} and \cite{Porwal14}
\item[6)] $T{{S}^{*}}C(\alpha ,0;1)=TC(\alpha )$ \cite{Silverman}
\end{itemize}

In this paper, we introduce and investigate two new subclasses ${{S}^{*}}C(\alpha ,\beta ;\gamma )$ and $T{{S}^{*}}C(\alpha ,\beta ;\gamma )$, $\alpha ,\beta \in \left[ 0,1 \right)$, $\gamma \in \left[ 0,1 \right]$of the analytic functions in the open unit disk. The object of the present paper is to derive characteristic properties of the functions belonging to these subclasses. Also, we examine some analytic functions which involve Gamma function, and provide conditions for these functions to be in these subclasses.

\section{\large Conditions for the subclasses ${{S}^{*}}C(\alpha ,\beta ;\gamma )$ and $T{{S}^{*}}C(\alpha ,\beta ;\gamma )$}
In this section, we will examine some characteristic properties of the subclasses ${{S}^{*}}C(\alpha ,\beta ;\gamma )$ and $T{{S}^{*}}C(\alpha ,\beta ;\gamma )$ of analytic functions in the open unit disk. \\
A sufficient condition for the functions in the subclass ${{S}^{*}}C(\alpha ,\beta ;\gamma )$ is given by the following theorem. 
\begin{thm}\label{thrm21}
Let $f\in A$. Then, the function $f(z)$ belongs to the class ${{S}^{*}}C(\alpha ,\beta ;\gamma )$, $\alpha ,\beta \in \left[ 0,1 \right),\gamma \in \left[ 0,1 \right]$ if the following condition is satisfied  \begin{equation}\label{eq21}
\sum\limits_{n=2}^{\infty }{\left( 1+(n-1)\gamma  \right)\left( n-\alpha -(n-1)\alpha \beta  \right)}\left| {{a}_{n}} \right|\le 1-\alpha .
\end{equation}
The result is sharp for the functions
\begin{equation}\label{eq22}
{{f}_{n}}(z)=z+\frac{1-\alpha }{\left( 1+(n-1)\gamma  \right)\left( n-\alpha -(n-1)\alpha \beta  \right)}{{z}^{n}},\text{ }z\in U,\text{ }n=2,3,...\text{ }.
\end{equation}
\end{thm}
\begin{proof}
From the Definition \ref{tnm11}, a function$f\in {{S}^{*}}C(\alpha ,\beta ;\gamma )$, $\alpha ,\beta \in \left[ 0,1 \right),\gamma \in \left[ 0,1 \right]$ if and only if 
\begin{equation}\label{eq23}
\operatorname{Re}\left\{ \frac{z{f}'(z)+\gamma {{z}^{2}}{f}''(z)}{\gamma z\left( {f}'(z)+\beta z{f}''(z) \right)+(1-\gamma )\left( \beta z{f}'(z)+(1-\beta )f(z) \right)} \right\}>\alpha .
\end{equation}
We can easily show that condition (\ref{eq23}) holds true if
$$\left| \frac{z{f}'(z)+\gamma {{z}^{2}}{f}''(z)}{\gamma z\left( {f}'(z)+\beta z{f}''(z) \right)+(1-\gamma )\left( \beta z{f}'(z)+(1-\beta )f(z) \right)}-1 \right|\le 1-\alpha .$$
Now, let us show that this condition is satisfied under the hypothesis (\ref{eq21}) of the theorem.
By simple computation, we write
\begin{align*}
  & \left| \frac{z{f}'(z)+\gamma {{z}^{2}}{f}''(z)}{\gamma z\left( {f}'(z)+\beta z{f}''(z) \right)+(1-\gamma )\left( \beta z{f}'(z)+(1-\beta )f(z) \right)}-1 \right| \\ 
 & =\left| \frac{\sum\limits_{n=2}^{\infty }{\left( 1+(n-1)\gamma  \right)(n-1)(1-\beta ){{a}_{n}}{{z}^{n}}}}{z+\sum\limits_{n=2}^{\infty }{\left( 1+(n-1)\gamma  \right)\left( 1+(n-1)\beta  \right){{a}_{n}}{{z}^{n}}}} \right|\le \frac{\sum\limits_{n=2}^{\infty }{\left( 1+(n-1)\gamma  \right)(n-1)(1-\beta )\left| {{a}_{n}} \right|}}{1-\sum\limits_{n=2}^{\infty }{\left( 1+(n-1)\gamma  \right)\left( 1+(n-1)\beta  \right)\left| {{a}_{n}} \right|}}.  
\end{align*}
Last expression of the above inequality is bounded by $1-\alpha $ if
$$\sum\limits_{n=2}^{\infty }{\left( 1+(n-1)\gamma  \right)(n-1)(1-\beta )\left| {{a}_{n}} \right|}\le (1-\alpha )\left\{ 1-\sum\limits_{n=2}^{\infty }{\left( 1+(n-1)\gamma  \right)\left( 1+(n-1)\beta  \right)\left| {{a}_{n}} \right|} \right\},$$
which is equivalent to
$$\sum\limits_{n=2}^{\infty }{\left( 1+(n-1)\gamma  \right)\left( n-\alpha -(n-1)\alpha \beta  \right)}\left| {{a}_{n}} \right|\le 1-\alpha .$$
Also, we can easily see that the equality in (\ref{eq21}) is satisfied by the functions given by (\ref{eq22}). \\
Thus, the proof of Theorem \ref{thrm21} is completed.
\end{proof} 
Setting $\gamma =0$ and $\gamma =1$ in Theorem \ref{thrm21}, we can readily deduce the following results, respectively.

\begin{cor}\label{snc21}
The function $f(z)$ definition by (\ref{eq11}) belongs to the class ${{S}^{*}}(\alpha ,\beta )$, $\alpha ,\beta \in \left[ 0,1 \right)$ if the following condition is satisfied  
$$\sum\limits_{n=2}^{\infty }{\left( n-\alpha -(n-1)\alpha \beta  \right)}\left| {{a}_{n}} \right|\le 1-\alpha. $$
The result is sharp for the functions
$${{f}_{n}}(z)=z+\frac{1-\alpha }{n-\alpha -(n-1)\alpha \beta }{{z}^{n}},\text{ }z\in U,\text{ }n=2,3,...\text{ }.$$ 
\end{cor}

\begin{cor}\label{snc22}
The function $f(z)$ definition by (\ref{eq11}) belongs to the class $C(\alpha ,\beta )$, $\alpha ,\beta \in \left[ 0,1 \right)$ if the following condition is satisfied   
$$\sum\limits_{n=2}^{\infty }{n\left( n-\alpha -(n-1)\alpha \beta  \right)}\left| {{a}_{n}} \right|\le 1-\alpha .$$ 
The result is sharp for the functions 
$${{f}_{n}}(z)=z+\frac{1-\alpha }{n\left( n-\alpha -(n-1)\alpha \beta  \right)}{{z}^{n}},\text{ }z\in U,\text{ }n=2,3,...\text{ }.$$
\end{cor}

\begin{cor}\label{snc23}
The function $f(z)$ definition by (\ref{eq11}) belongs to the class ${{S}^{*}}(\alpha )$, $\alpha \in \left[ 0,1 \right)$ if the following condition is satisfied     
$$\sum\limits_{n=2}^{\infty }{\left( n-\alpha  \right)}\left| {{a}_{n}} \right|\le 1-\alpha .$$ 
The result is sharp for the functions 
$${{f}_{n}}(z)=z+\frac{1-\alpha }{n-\alpha }{{z}^{n}},\text{ }z\in U,\text{ }n=2,3,...\text{ }.$$
\end{cor}

\begin{cor}\label{snc24}
The function $f(z)$ definition by (\ref{eq11}) belongs to the class $C(\alpha )$, $\alpha \in \left[ 0,1 \right)$ if the following condition is satisfied      
$$\sum\limits_{n=2}^{\infty }{n\left( n-\alpha  \right)}\left| {{a}_{n}} \right|\le 1-\alpha. $$ 
The result is sharp for the functions 
$${{f}_{n}}(z)=z+\frac{1-\alpha }{n\left( n-\alpha  \right)}{{z}^{n}},\text{ }z\in U,\text{ }n=2,3,...\text{ }.$$
\end{cor}
\begin{rem}\label{yrm21}
Further consequences of the properties given by Corollary \ref{snc23} and Corollary \ref{snc24} can be obtained for each of the classes studied by earlier researches, by specializing the various parameters involved. Many of these consequences were proved by earlier researches on the subject (cf., e.g., \cite{Silverman}).
\end{rem}
For the function in the class $T{{S}^{*}}C(\alpha ,\beta ;\gamma )$, the converse of Theorem \ref{thrm21} is also true.
\begin{thm}\label{thrm22}
Let $f\in T$. Then, the function $f(z)$ belongs to the class $T{{S}^{*}}C(\alpha ,\beta ;\gamma )$, $\alpha ,\beta \in \left[ 0,1 \right),\text{ }\gamma \in \left[ 0,1 \right]$ if and only if   
\begin{equation}\label{eq24}
\sum\limits_{n=2}^{\infty }{\left( 1+(n-1)\gamma  \right)\left( n-\alpha -(n-1)\alpha \beta  \right)}{{a}_{n}}\le 1-\alpha .
\end{equation}
The result is sharp for the functions
\begin{equation}\label{eq25}
{{f}_{n}}(z)=z-\frac{1-\alpha }{\left( 1+(n-1)\gamma  \right)\left( n-\alpha -(n-1)\alpha \beta  \right)}{{z}^{n}},\text{ }z\in U,\text{ }n=2,3,...\text{ }.
\end{equation}
\end{thm}
\begin{proof}
The proof of the sufficiency of the theorem can be proved similarly to the proof of Theorem \ref{thrm21}. \\
We will prove only the necessity of the theorem.  \\
Assume that $f\in T{{S}^{*}}C(\alpha ,\beta ;\gamma )$, $\alpha ,\beta \in \left[ 0,1 \right),\text{ }\gamma \in \left[ 0,1 \right]$; that is, 
$$\operatorname{Re}\left\{ \frac{z{f}'(z)+\gamma {{z}^{2}}{f}''(z)}{\gamma z\left( {f}'(z)+\beta z{f}''(z) \right)+(1-\gamma )\left( \beta z{f}'(z)+(1-\beta )f(z) \right)} \right\}>\alpha ,\text{ }z\in U.$$
By simple computation, we write
\begin{align*}
  & \operatorname{Re}\left\{ \frac{z{f}'(z)+\gamma {{z}^{2}}{f}''(z)}{\gamma z\left( {f}'(z)+\beta z{f}''(z) \right)+(1-\gamma )\left( \beta z{f}'(z)+(1-\beta )f(z) \right)} \right\} \\ 
 & \text{                   }=\operatorname{Re}\left\{ \frac{z-\sum\limits_{n=2}^{\infty }{n\left( 1+(n-1)\gamma  \right){{a}_{n}}}{{z}^{n}}}{z-\sum\limits_{n=2}^{\infty }{\left( 1+(n-1)\gamma  \right)\left( 1+(n-1)\beta  \right){{a}_{n}}}{{z}^{n}}} \right\}>\alpha .
\end{align*}
The last expression in the brackets of the above inequality is real if choose $z$ real. Hence, from the previous inequality letting $z\to 1$ through real values, we obtain
$$1-\sum\limits_{n=2}^{\infty }{n\left( 1+(n-1)\gamma  \right){{a}_{n}}}\ge \alpha \left\{ 1-\sum\limits_{n=2}^{\infty }{\left( 1+(n-1)\gamma  \right)\left( 1+(n-1)\beta  \right){{a}_{n}}} \right\}.$$
This follows
$$\sum\limits_{n=2}^{\infty }{\left( 1+(n-1)\gamma  \right)\left( n-\alpha -(n-1)\alpha \beta  \right)}{{a}_{n}}\le 1-\alpha ,$$
which is the same as the condition (\ref{eq24}). \\
Also, it is clear that the equality in (\ref{eq24}) is satisfied by the functions given by (\ref{eq25}). \\
Thus, the proof of Theorem \ref{thrm22} is completed. 
\end{proof}
Taking $\gamma =0$ and $\gamma =1$ in Theorem \ref{thrm22}, we can readily deduce the following results, respectively.
\begin{cor}\label{snc25}
The function $f(z)$ definition by (\ref{eq12}) belongs to the class $T{{S}^{*}}(\alpha ,\beta )$, $\alpha ,\beta \in \left[ 0,1 \right)$ if and only if     
$$\sum\limits_{n=2}^{\infty }{\left( n-\alpha -(n-1)\alpha \beta  \right)}{{a}_{n}}\le 1-\alpha .$$ 
The result is sharp for the functions 
$${{f}_{n}}(z)=z-\frac{1-\alpha }{n-\alpha -(n-1)\alpha \beta }{{z}^{n}},\text{ }z\in U,\text{ }n=2,3,...\text{ }.$$
\end{cor}

\begin{cor}\label{snc26}
The function   definition by (\ref{eq12}) belongs to the class $TC(\alpha ,\beta )$, $\alpha ,\beta \in \left[ 0,1 \right)$ if and only if    
$$\sum\limits_{n=2}^{\infty }{n\left( n-\alpha -(n-1)\alpha \beta  \right)}{{a}_{n}}\le 1-\alpha. $$
The result is sharp for the functions
$${{f}_{n}}(z)=z-\frac{1-\alpha }{n\left( n-\alpha -(n-1)\alpha \beta  \right)}{{z}^{n}},\text{ }z\in U,\text{ }n=2,3,...\text{ }.$$ 
\end{cor}
Setting $\beta =0$ in Corollary \ref{snc25} and \ref{snc26}, we can readily deduce the following results, respectively.
\begin{cor}\label{snc27}
The function $f(z)$ definition by (\ref{eq12}) belongs to the class $T{{S}^{*}}(\alpha )$, $\alpha \in \left[ 0,1 \right)$ if and only if       
$$\sum\limits_{n=2}^{\infty }{\left( n-\alpha  \right)}{{a}_{n}}\le 1-\alpha .$$ 
The result is sharp for the functions 
$${{f}_{n}}(z)=z-\frac{1-\alpha }{n-\alpha }{{z}^{n}},\text{ }z\in U,\text{ }n=2,3,...\text{ }.$$
\end{cor}

\begin{cor}\label{snc28}
The function $f(z)$ definition by (\ref{eq12}) belongs to the class $TC(\alpha )$, $\alpha \in \left[ 0,1 \right)$ if and only if        
$$\sum\limits_{n=2}^{\infty }{n\left( n-\alpha  \right)}{{a}_{n}}\le 1-\alpha .$$ 
The result is sharp for the functions 
$${{f}_{n}}(z)=z-\frac{1-\alpha }{n\left( n-\alpha  \right)}{{z}^{n}},\text{ }z\in U,\text{ }n=2,3,...\text{ }.$$
\end{cor}
\begin{rem}\label{yrm22}
The results obtained by Corollary \ref{snc27} and Corollary \ref{snc28} would reduce to known results in \cite{Alt?ntas88}. 
\end{rem}
\begin{rem}\label{yrm23}
Further consequences of the properties given by Corollary \ref{snc27} and Corollary \ref{snc28} can be obtained for each of the classes studied by earlier researches, by specializing the various parameters involved. Many of these consequences were proved by earlier researches on the subject (cf., e.g., \cite{Silverman}).
\end{rem}
From Theorem \ref{thrm22}, we obtain the following theorem on the coefficient bound estimates.
\begin{thm}\label{thrm23}
Let the function definition by (\ref{eq12}) $f\in T{{S}^{*}}C(\alpha ,\beta ;\gamma )$, $\alpha ,\beta \in \left[ 0,1 \right),\gamma \in \left[ 0,1 \right]$. Then  
$${{a}_{n}}\le \frac{1-\alpha }{\left( 1+(n-1)\gamma  \right)\left( n-\alpha -(n-1)\alpha \beta  \right)},n=2,3,...~.$$ 
\end{thm}
\begin{cor}\label{snc29}
Let the function definition by (\ref{eq12}) $f\in T{{S}^{*}}(\alpha ,\beta )$, $\alpha ,\beta \in \left[ 0,1 \right)$. Then  
$${{a}_{n}}\le \frac{1-\alpha }{n-\alpha -(n-1)\alpha \beta },n=2,3,...~.$$ 
\end{cor}

\begin{cor}\label{snc210}
Let the function definition by (\ref{eq12}) $f\in TC(\alpha ,\beta )$, $\alpha ,\beta \in \left[ 0,1 \right)$. Then   
$${{a}_{n}}\le \frac{1-\alpha }{n\left( n-\alpha -(n-1)\alpha \beta  \right)},n=2,3,...\text{ }.$$ 
\end{cor}
Setting $\beta =0$ in Corollary \ref{snc29} and \ref{snc210}, we can readily deduce the following results, respectively.
\begin{cor}\label{snc211}
Let the function definition by (\ref{eq12}) $f\in T{{S}^{*}}(\alpha )$, $\alpha \in \left[ 0,1 \right)$. Then   
$${{a}_{n}}\le \frac{1-\alpha }{n-\alpha },n=2,3,... \text{ }.$$
\end{cor}

\begin{cor}\label{snc212}
Let the function definition by (\ref{eq12}) $f\in TC(\alpha )$, $\alpha \in \left[ 0,1 \right)$. Then    
$${{a}_{n}}\le \frac{1-\alpha }{n\left( n-\alpha  \right)},n=2,3,... \text{ }.$$
\end{cor}
\begin{rem}\label{yrm24}
Further consequences on the coefficient bound estimates given by Corollary \ref{snc211} and Corollary \ref{snc212} can be obtained for each of the classes studied by earlier researches, by specializing the various parameters involved. 
\end{rem}

\section{\large Conditions for the analytic functions involving Gamma\\ function}
In this section, we will examine geometric properties of analytic functions involving Gamma function. For these functions, we give conditions to be in these classes ${{S}^{*}}C(\alpha ,\beta ;\gamma )$ and $T{{S}^{*}}C(\alpha ,\beta ;\gamma )$. \\
Let us define the function ${{F}_{\lambda ,\mu }}:\mathbb{C}\to \mathbb{C}$ by 

\begin{align}\label{eq31}
  & {{F}_{\lambda ,\mu }}(z)=z+\sum\limits_{n=2}^{\infty }{\frac{\Gamma (\mu )}{\Gamma (\lambda (n-1)+\mu )}\frac{{{e}^{-1/\mu }}}{(n-1)!}{{z}^{n}}=z+\left( {{W}_{\lambda ,\mu }}(z)-z \right){{e}^{-1/\mu }},\text{ }} \\ \nonumber
 & z\in U,\lambda >-1,\mu >0, 
\end{align}
where $\Gamma (\mu )$ is Euler gamma function and ${{W}_{\lambda ,\mu }}(z)$ is normalized Wright function (see, for details \cite{Wright}). \\
We define also the function                                 
\begin{equation}\label{eq32}
{{G}_{\lambda ,\mu }}(z)=2z-{{F}_{\lambda ,\mu }}(z)=z-\sum\limits_{n=2}^{\infty }{\frac{\Gamma (\mu )}{\Gamma (\lambda (n-1)+\mu )}\frac{{{e}^{-1/\mu }}}{(n-1)!}{{z}^{n}},\text{ }z\in U}.
\end{equation}
It is clear that ${{F}_{\lambda ,\mu }}\in A$ and ${{G}_{\lambda ,\mu }}\in T$, respectively. \\
We will give sufficient condition for the function ${{F}_{\lambda ,\mu }}(z)$ defined by (\ref{eq31}), belonging to the class ${{S}^{*}}C(\alpha ,\beta ;\gamma )$, and necessary and sufficient condition for the function ${{G}_{\lambda ,\mu }}(z)$ defined by (\ref{eq32}), belonging to the class $T{{S}^{*}}C(\alpha ,\beta ;\gamma )$, respectively. 
\begin{thm}\label{thrm31}
Let $\lambda \ge 1,\mu >0.462$ and the following condition is satisfied 
\begin{equation}\label{eq33}
\left\{ \left( 1-\alpha \beta  \right)\gamma +\left[ 1-\alpha \beta +\left( 2-(1+\beta )\alpha  \right)\gamma  \right]\mu  \right\}{{\mu }^{-2}}{{e}^{1/\mu }}\le 1-\alpha .
\end{equation}
Then, the function ${{F}_{\lambda ,\mu }}(z)$defined by (\ref{eq31}) is in the class ${{S}^{*}}C(\alpha ,\beta ;\gamma ),~\alpha ,\beta \in \left[ 0,1 \right),\gamma \in \left[ 0,1 \right]$.
\end{thm}
\begin{proof}
Since ${{F}_{\lambda ,\mu }}\in A$ and 
$${{F}_{\lambda ,\mu }}(z)=z+\sum\limits_{n=2}^{\infty }{\frac{\Gamma (\mu )}{\Gamma (\lambda (n-1)+\mu )}\frac{{{e}^{-1/\mu }}}{(n-1)!}{{z}^{n}}\text{ }},$$
in view of Theorem \ref{thrm21}, it suffices show that 
\begin{equation}\label{eq34}
\sum\limits_{n=2}^{\infty }{\left( 1+(n-1)\gamma  \right)\left( n-\alpha -(n-1)\alpha \beta  \right)}\frac{\Gamma (\mu )}{\Gamma (\lambda (n-1)+\mu )}\frac{{{e}^{-1/\mu }}}{(n-1)!}\le 1-\alpha .
\end{equation}
Let
$${{L}_{1}}(p;\alpha ,\beta ;\gamma )=\sum\limits_{n=2}^{\infty }{\left( 1+(n-1)\gamma  \right)\left( n-\alpha -(n-1)\alpha \beta  \right)}\frac{\Gamma (\mu )}{\Gamma (\lambda (n-1)+\mu )}\frac{{{e}^{-1/\mu }}}{(n-1)!}.$$
Under hypothesis $\lambda \ge 1$, the inequality $\Gamma (n-1+\mu )\le \Gamma (\lambda (n-1)+\mu ),\text{ }n\in \mathbb{N}$ holds for $\mu >0.462$, which is equivalent to 
\begin{equation}\label{eq35}
\frac{\Gamma (\mu )}{\Gamma (\lambda (n-1)+\mu )}\le \frac{1}{{{(\mu )}_{n-1}}},\text{ }n\in \mathbb{N}
\end{equation}
where ${{(\mu )}_{n}}=\Gamma (n+\mu )/\Gamma (\mu )=\mu (\mu +1)\cdot \cdot \cdot (\mu +n-1),\text{ }{{(\mu )}_{0}}=1$ is Pochhammer (or Appell) symbol, defined in terms of Euler gamma function.\\
Aslo, the inequality
\begin{equation}\label{eq36}
{{(\mu )}_{n-1}}=\mu (\mu +1)\cdot \cdot \cdot (\mu +n-2)\ge {{\mu }^{n-1}},\text{ }n\in \mathbb{N}
\end{equation}
is true, which is equivalent to $1/{{(\mu )}_{n-1}}\le 1/{{\mu }^{n-1}},\text{ }n\in \mathbb{N}$. \\
Setting 
\begin{align*}
  & \left( 1+(n-1)\gamma  \right)\left( n-\alpha -(n-1)\alpha \beta  \right) \\ 
 & \text{             }=(n-2)(n-1)\left( 1-\alpha \beta  \right)\gamma +(n-1)\left( 1-\alpha \beta +\left( 2-(1+\beta )\alpha  \right)\gamma  \right)+1-\alpha  
\end{align*}
and using (\ref{eq35}), (\ref{eq36}), we can easily write that
\begin{align*}
  & {{L}_{1}}(p;\alpha ,\beta ;\gamma )\le  \\ 
 & \sum\limits_{n=2}^{\infty }{\left\{ (n-2)(n-1)\left( 1-\alpha \beta  \right)\gamma +(n-1)\left( 1-\alpha \beta +\left( 2-(1+\beta )\alpha  \right)\gamma  \right)+1-\alpha  \right\}}\frac{{{e}^{-1/\mu }}}{{{\mu }^{n-1}}(n-1)!} \\ 
 & =\sum\limits_{n=3}^{\infty }{\frac{\left( 1-\alpha \beta  \right)\gamma }{(n-3)!}\frac{{{e}^{-1/\mu }}}{{{\mu }^{n-1}}}}+\sum\limits_{n=2}^{\infty }{\frac{1-\alpha \beta +\left( 2-(1+\beta )\alpha  \right)\gamma }{(n-2)!}\frac{{{e}^{-1/\mu }}}{{{\mu }^{n-1}}}+\sum\limits_{n=2}^{\infty }{\frac{(1-\alpha ){{e}^{-1/\mu }}}{{{\mu }^{n-1}}(n-1)!}}}  
\end{align*}
Thus, 
$${{L}_{1}}(p;\alpha ,\beta ;\gamma )\le \frac{\left( 1-\alpha \beta  \right)\gamma }{{{\mu }^{2}}}+\frac{1-\alpha \beta +\left( 2-(1+\beta )\alpha  \right)\gamma }{\mu }+(1-\alpha )\left( 1-{{e}^{-1/\mu }} \right).$$
From the last inequality we easily see that the inequality (\ref{eq34}) is true if last expression is bounded by $1-\alpha $, which is equivalent to (\ref{eq33}). \\
Thus, the proof of Theorem \ref{thrm31} is completed.
\end{proof}
Taking $\gamma =0$ and $\gamma =1$ in Theorem \ref{thrm31}, we arrive at the following results.
\begin{cor}\label{snc31}
Let $\lambda \ge 1,\mu >0.462$ and the following condition is satisfied
$$\left( 1-\alpha \beta  \right){{\mu }^{-1}}{{e}^{1/\mu }}\le 1-\alpha .$$
Then, the function ${{F}_{\lambda ,\mu }}(z)$ defined by (\ref{eq31}) is in the class ${{S}^{*}}(\alpha ,\beta ),~\alpha ,\beta \in \left[ 0,1 \right)$.
\end{cor}

\begin{cor}\label{snc32}
Let $\lambda \ge 1,\mu >0.462$ and the following condition is satisfied
$$\left\{ \left( 1-\alpha \beta  \right){{\mu }^{-2}}+\left( 3-2\alpha \beta -\alpha  \right){{\mu }^{-1}} \right\}{{e}^{1/\mu }}\le 1-\alpha  .$$
Then, the function ${{F}_{\lambda ,\mu }}(z)$ defined by (\ref{eq31}) is in the class $C(\alpha ,\beta ),\alpha ,\beta \in \left[ 0,1 \right)$.
\end{cor}
\begin{rem}\label{yrm31}
Further consequences of the results given by Corollary \ref{snc31} and Corollary \ref{snc32} can be obtained for each of the classes, by specializing the various parameters involved.
\end{rem}
\begin{thm}\label{thrm32}
Let $\lambda \ge 1,\mu >0.462$, then the function ${{G}_{\lambda ,\mu }}(z)$ defined by (3.2) belongs to the class $T{{S}^{*}}C(\alpha ,\beta ;\gamma )$, $\alpha ,\beta \in \left[ 0,1 \right),\gamma \in \left[ 0,1 \right]$ if 
\begin{equation}\label{eq37}
\left\{ \left( 1-\alpha \beta  \right)\gamma +\left[ 1-\alpha \beta +\left( 2-(1+\beta )\alpha  \right)\gamma  \right]\mu  \right\}{{\mu }^{-2}}{{e}^{1/\mu }}\le 1-\alpha .
\end{equation}
\end{thm}
\begin{proof}
The proof of Theorem \ref{thrm32} is same of the proof of Theorem \ref{thrm31}. Therefore, the details of the proof of Theorem \ref{thrm32} may be omitted.
\end{proof}
\begin{rem}\label{yrm32}
Further consequences of the results given by Theorem \ref{thrm32} can be obtained for each of the classes, by specializing the various parameters involved.
\end{rem}
\section{\large Integral operators of the functions ${{F}_{\lambda ,\mu }}(z)$ and ${{G}_{\lambda ,\mu }}(z)$}

In this section, we will examine some inclusion properties of integral operators associated with the functions ${{F}_{\lambda ,\mu }}(z)$ and ${{G}_{\lambda ,\mu }}(z)$ as follows:
\begin{equation}\label{eq41}
{{\hat{F}}_{\lambda \mu }}(z)=\int\limits_{0}^{z}{\frac{{{F}_{\lambda \mu }}(t)}{t}dt} \text{  and  } {{\hat{G}}_{\lambda ,\mu }}(z)=\int\limits_{0}^{z}{\frac{{{G}_{\lambda ,\mu }}(t)}{t}dt}
\end{equation}

\begin{thm}\label{thrm41}
Let $\lambda \ge 1,\mu >0.462$ and the following condition is satisfied 
\begin{equation}\label{eq42}
\left\{ \left( 1-\alpha \beta  \right)\gamma {{\mu }^{-1}}+\left( 1-\beta  \right)\left( 1-\gamma  \right)\alpha \left( 1-{{e}^{-1/\mu }} \right) \right\}{{e}^{1/\mu }}\le 1-\alpha .
\end{equation}
Then, the function ${{\hat{F}}_{\lambda ,\mu }}(z)$ defined by (\ref{eq41}) is in the class ${{S}^{*}}C(\alpha ,\beta ;\gamma ),~\alpha ,\beta \in \left[ 0,1 \right),\gamma \in \left[ 0,1 \right]$.
\end{thm}
\begin{proof}
Since 
$${{\hat{F}}_{\lambda ,\mu }}(z)=z+\sum\limits_{n=2}^{\infty }{\frac{\Gamma (\mu )}{\Gamma (\lambda (n-1)+\mu )}\frac{{{e}^{-1/\mu }}}{n!}{{z}^{n}},z\in U\text{ }}$$
according to Theorem \ref{thrm21}, the function ${{\hat{F}}_{\lambda ,\mu }}(z)$ belongs to the class ${{S}^{*}}C(\alpha ,\beta ;\gamma )$ if the following condition is satisfied   
\begin{equation}\label{eq43}
\sum\limits_{n=2}^{\infty }{\left( 1+(n-1)\gamma  \right)\left( n-\alpha -(n-1)\alpha \beta  \right)}\frac{\Gamma (\mu )}{\Gamma (\lambda (n-1)+\mu )}\frac{{{e}^{-1/\mu }}}{n!}\le 1-\alpha .
\end{equation}
Let 
$${{L}_{2}}(p;\alpha ,\beta ;\gamma )=\sum\limits_{n=2}^{\infty }{\left( 1+(n-1)\gamma  \right)\left( n-\alpha -(n-1)\alpha \beta  \right)}\frac{\Gamma (\mu )}{\Gamma (\lambda (n-1)+\mu )}\frac{{{e}^{-1/\mu }}}{n!}.$$
Setting 
\begin{align*}
  & \left( 1+(n-1)\gamma  \right)\left( n-\alpha -(n-1)\alpha \beta  \right) \\ 
 & \text{                }=(n-1)n\left( 1-\alpha \beta  \right)\gamma +n\left( (1-\alpha \beta )(1-\gamma )+(1-\alpha )\gamma  \right)-(1-\beta )(1-\gamma )\alpha   
\end{align*}
and by simple computation, we obtain
\begin{align*}
  & {{L}_{2}}(p;\alpha ,\beta ;\gamma )= \\ 
 & \sum\limits_{n=2}^{\infty }\left\{ (n-1)n\left( 1-\alpha \beta  \right)\gamma +n(1-\beta )(1-\gamma )\alpha\right. \\
 & \left. -(1-\beta )(1-\gamma )\alpha +n(1-\alpha ) \right\}\frac{\Gamma (\mu )}{\Gamma (\lambda (n-1)+\mu )}\frac{{{e}^{-1/\mu }}}{n!} \\ 
 & =\sum\limits_{n=2}^{\infty }{\frac{\left( 1-\alpha \beta  \right)\gamma }{(n-2)!}\frac{\Gamma (\mu ){{e}^{-1/\mu }}}{\Gamma (\lambda (n-1)+\mu )}}+\sum\limits_{n=2}^{\infty }{\frac{(1-\beta )(1-\gamma )\alpha }{(n-1)!}\frac{\Gamma (\mu ){{e}^{-1/\mu }}}{\Gamma (\lambda (n-1)+\mu )}-} \\ 
 & \sum\limits_{n=2}^{\infty }{\frac{(1-\beta )(1-\gamma )\alpha }{n!}\frac{\Gamma (\mu ){{e}^{-1/\mu }}}{\Gamma (\lambda (n-1)+\mu )}+\sum\limits_{n=2}^{\infty }{\frac{1-\alpha }{(n-1)!}\frac{\Gamma (\mu ){{e}^{-1/\mu }}}{\Gamma (\lambda (n-1)+\mu )}}}. 
\end{align*}
Thus,
\begin{align*}
  & {{L}_{2}}(p;\alpha ,\beta ;\gamma )\le \sum\limits_{n=2}^{\infty }{\frac{\left( 1-\alpha \beta  \right)\gamma }{(n-2)!}\frac{\Gamma (\mu ){{e}^{-1/\mu }}}{\Gamma (\lambda (n-1)+\mu )}}+\sum\limits_{n=2}^{\infty }{\frac{(1-\beta )(1-\gamma )\alpha }{(n-1)!}\frac{\Gamma (\mu ){{e}^{-1/\mu }}}{\Gamma (\lambda (n-1)+\mu )}}+ \\ 
 & \text{                         }\sum\limits_{n=2}^{\infty }{\frac{1-\alpha }{(n-1)!}\frac{\Gamma (\mu ){{e}^{-1/\mu }}}{\Gamma (\lambda (n-1)+\mu )}}. 
\end{align*}
From (\ref{eq35}) and (\ref{eq36}), w have
\begin{align*}
  & {{L}_{2}}(p;\alpha ,\beta ;\gamma )\le \sum\limits_{n=2}^{\infty }{\frac{\left( 1-\alpha \beta  \right)\gamma }{(n-2)!}\frac{{{e}^{-1/\mu }}}{{{\mu }^{n-1}}}}+\sum\limits_{n=2}^{\infty }{\frac{(1-\beta )(1-\gamma )\alpha }{(n-1)!}\frac{{{e}^{-1/\mu }}}{{{\mu }^{n-1}}}}+\sum\limits_{n=2}^{\infty }{\frac{1-\alpha }{(n-1)!}\frac{{{e}^{-1/\mu }}}{{{\mu }^{n-1}}}}= \\ 
 & \frac{\left( 1-\alpha \beta  \right)\gamma }{\mu }+(1-\beta )(1-\gamma )\alpha \left( 1-{{e}^{-1/\mu }} \right)+(1-\alpha )\left( 1-{{e}^{-1/\mu }} \right).
\end{align*}
Therefore, inequality (\ref{eq43}) holds true if
$$\frac{\left( 1-\alpha \beta  \right)\gamma }{\mu }+(1-\beta )(1-\gamma )\alpha \left( 1-{{e}^{-1/\mu }} \right)+(1-\alpha )\left( 1-{{e}^{-1/\mu }} \right)\le 1-\alpha ,$$
which is equivalent to (\ref{eq42}).\\
Thus, the proof of Theorem \ref{thrm41} is completed. 
\end{proof}
Taking $\gamma =0$ and $\gamma =1$ in Theorem \ref{thrm41}, we arrive at the following results.
\begin{cor}\label{snc41}
Let $\lambda \ge 1,\mu >0.462$ and the following condition is satisfied
$$\left( 1-\beta  \right)\alpha \left( {{e}^{1/\mu }}-1 \right)\le 1-\alpha  .$$
Then, the function ${{\hat{F}}_{\lambda ,\mu }}(z)$ defined by (\ref{eq41}) is in the class ${{S}^{*}}(\alpha ,\beta ),\alpha ,\beta \in \left[ 0,1 \right)$.
\end{cor}

\begin{cor}\label{snc42}
Let $\lambda \ge 1,\mu >0.462$ and the following condition is satisfied
$$\left( 1-\alpha \beta  \right){{\mu }^{-1}}{{e}^{1/\mu }}\le 1-\alpha  .$$
Then, the function ${{\hat{F}}_{\lambda ,\mu }}(z)$ defined by (\ref{eq41}) is in the class $C(\alpha ,\beta ),\alpha ,\beta \in \left[ 0,1 \right)$.
\end{cor}
\begin{rem}\label{yrm41}
Further consequences of the results given by Corollary \ref{snc41} and Corollary \ref{snc42} can be obtained for each of the classes, by specializing the various parameters involved.
\end{rem}
\begin{thm}\label{thrm42}
Let $\lambda \ge 1,\mu >0.462$, then the function ${{\hat{G}}_{\lambda ,\mu }}(z)$ defined by (\ref{eq41}) belongs to the class $T{{S}^{*}}C(\alpha ,\beta ;\gamma )$, $\alpha ,\beta \in \left[ 0,1 \right),\gamma \in \left[ 0,1 \right]$ if  
$$\left\{ \left( 1-\alpha \beta  \right)\gamma {{\mu }^{-1}}+\left( 1-\beta  \right)\left( 1-\gamma  \right)\alpha \left( 1-{{e}^{-1/\mu }} \right) \right\}{{e}^{1/\mu }}\le 1-\alpha .$$
\end{thm}
\begin{proof}
The proof of Theorem \ref{thrm42} is same of the proof of Theorem \ref{thrm41}. Therefore, the details of the proof of Theorem \ref{thrm42} may be omitted.
\end{proof}
Taking $\gamma =0$ and $\gamma =1$ in Theorem \ref{thrm42}, we can readily deduce the following results, respectively.

\begin{cor}\label{snc43}
Let $\lambda \ge 1,\mu >0.462$, then the function ${{\hat{G}}_{\lambda ,\mu }}(z)$ defined by (\ref{eq41}) belongs to the class $T{{S}^{*}}(\alpha ,\beta )$, $\alpha ,\beta \in \left[ 0,1 \right)$ if 
$$\left( 1-\beta  \right)\alpha \left( {{e}^{1/\mu }}-1 \right)\le 1-\alpha .$$
\end{cor}
\begin{cor}\label{snc44}
Let $\lambda \ge 1,\mu >0.462$, then the function ${{\hat{G}}_{\lambda ,\mu }}(z)$ defined by (4.1) belongs to the class $TC(\alpha ,\beta )$, $\alpha ,\beta \in \left[ 0,1 \right)$ if 
$$\left( 1-\alpha \beta  \right){{\mu }^{-1}}{{e}^{1/\mu }}\le 1-\alpha .$$
\end{cor}

\begin{rem}\label{yrm42}
Further consequences of the results given by Corollary \ref{snc43} and Corollary \ref{snc44} can be obtained for each of the classes, by specializing the various parameters involved. 
\end{rem}

\section{Conclusions and Discussions}
In this paper, we defined two general subclasses of the analytic functions in the open unit disk in the complex plane. We obtained coefficient estimates for these functions in these classes. From these results, we can easily obtain results found in the literature (See \cite{Alt?ntas91,Alt?ntas88,Silverman}). \\ 
Moreover, in this paper, analytic functions involving the Gamma function and their integral operators were investigated. The sufficient and also necessary and sufficient conditions for these functions to be in the classes ${{S}^{*}}C(\alpha ,\beta ;\gamma )$ and $T{{S}^{*}}C(\alpha ,\beta ;\gamma )$ are given.

\begin{note}
The short abstract of this manuscript was previously presented and published in \textit{ICANAS-2017, Antalya, TURKEY} (See \cite{Mustafa}).
\end{note}

\end{document}